\documentclass[11pt]{article}

\usepackage[margin=1in]{geometry}
\usepackage{amsmath,amssymb,amsthm,mathtools}
\usepackage{microtype}
\usepackage{enumitem}
\usepackage[hidelinks]{hyperref}

\allowdisplaybreaks
\setlist[enumerate]{leftmargin=*,itemsep=2pt,topsep=4pt}

\newtheorem{theorem}{Theorem}[section]

\newtheorem{lemma}[theorem]{Lemma}
\newtheorem{corollary}[theorem]{Corollary}
\theoremstyle{definition}
\newtheorem{definition}[theorem]{Definition}
\theoremstyle{remark}
\newtheorem{remark}[theorem]{Remark}

\newcommand{\R}{\mathbb R}
\newcommand{\rank}{\operatorname{rank}}
\newcommand{\dotcupunion}{\mathbin{\dot\cup}}
\newcommand{\1}{\mathbf 1}

\title{Incidence Rank and Bounded Defect\\
for Linear Hypergraphs with Cograph Line Graphs}
\author{Mahesh Ramani}
\date{16 August 2026}

\begin{document}
\maketitle

\begin{abstract}
Let \(N(H)\) be the edge--vertex incidence matrix of a finite linear
\(s\)-uniform hypergraph whose line graph is a cograph, and define
\[
 R_s(H)=(s+1)\operatorname{rank}_{\mathbb R}N(H)-s|E(H)|.
\]
The inequality \(R_s(H)\ge0\) holds, with equality characterized by affine
planes of order \(s\) on every row-dependent component.  Under the corresponding
finite-net completion hypothesis, connected systems with bounded \(R_s\) admit
an explicit classification.

If \(F\) is connected, \(R_s(F)=d\le s-1\), the completion hypothesis
holds through deficiency \(d\), and \(s>(d-1)^2\), then either \(N(F)\)
has full row rank and \(|E(F)|=d\), or \(F\) lies in a unique affine plane
of order \(s\).  In the latter case, for some \(0\le h\le d\), the system
contains every line in \(s+1-h\) parallel classes and exactly \(d-h\)
additional lines from the remaining \(h\) classes.  Conversely, each
such system has defect \(d\).

Bruck's theorem gives the classification for every fixed \(d\) and all
sufficiently large \(s\), while Metsch's completion theorem covers
\(d=o(s^{1/3})\).  As a global consequence, if the total defect is at
most \(D\) in the completion range, deleting at most \(D\) edges leaves
a vertex-disjoint union of finite nets and affine planes.  The associated
column defect and row leverage are determined explicitly as well.
\end{abstract}

\section{Introduction}\label{sec:intro}

For a linear \(s\)-uniform hypergraph \(H\), the identity
\[
 N(H)N(H)^{\mathsf T}=sI+A(L(H))
\]
relates the real incidence rank to the spectrum of the line graph.  When
\(L(H)\) is a cograph, this relation yields
\[
 (s+1)\operatorname{rank}_{\mathbb R}N(H)\ge s|E(H)|.
\]
The row-dependent equality components are affine planes of order \(s\),
equivalently Steiner \(2\)-designs \(S(2,s,s^2)\).  Write \(P_4^s\) for a
linear \(s\)-uniform hypergraph on four edges whose line graph is an induced
four-vertex path.  The rank inequality then implies
\[
 |E(H)|\le\frac{s+1}{s}|V(H)|
\]
for every linear \(P_4^s\)-free \(s\)-graph, the extremal problem studied
by Adak and Verma \cite{AdakVerma2026}.

The central parameter is the integer slack
\[
 R_s(H)=(s+1)\operatorname{rank}_{\mathbb R}N(H)-s|E(H)|.
\]
The cograph join decomposition gives an exact recursion for \(R_s\).  A
join node of defect at most \(d\) contains at least \(s+1-d\) minimal
balanced co-components; these form parallel classes on a common
\(s^2\)-point set and hence a net of deficiency at most \(d\).  Classical
completion theorems then constrain the remaining edges.

Under the corresponding net-completion hypothesis, a connected
row-dependent system with \(R_s=d\) lies in a unique affine plane and
contains \(s+1-h\) complete directions together with \(d-h\) lines from
the remaining \(h\) directions, for some \(0\le h\le d\).  The only
alternative is full row rank, in which case the system has exactly \(d\)
edges.  Bruck's theorem supplies the completion hypothesis whenever
\[
 s>\frac12(d-1)^4+(d-1)^3+(d-1)^2+\frac34(d-1),
\]
and Metsch's theorem whenever
\[
 3s>8d^3-18d^2+8d+4.
\]
These bounds yield the component classification, a bounded edit result,
and explicit formulas for column defect and row leverage.

\section{Cographs and the incidence-rank inequality}\label{sec:rank}

For an \(s\)-uniform hypergraph \(H\), let \(N(H)\) be the edge--vertex
incidence matrix, with rows indexed by \(E(H)\), after isolated vertices
are removed.  Define
\[
 R_s(H)=(s+1)\operatorname{rank}_{\mathbb R}N(H)-s|E(H)|
\]
and
\[
 C(H)=|V^+(H)|-\operatorname{rank}_{\mathbb R}N(H),
\]
where \(V^+(H)\) is the set of nonisolated vertices.  The quantity \(R_s(H)\) is the \emph{incidence-rank defect}.  The identity
\begin{equation}\label{eq:rank-column-split}
 (s+1)|V^+(H)|-s|E(H)|=R_s(H)+(s+1)C(H).
\end{equation}

Let \(N(H)^\dagger\) denote the Moore--Penrose pseudoinverse.  If \(P_H=N(H)N(H)^\dagger\) is the orthogonal projector onto the column
space of \(N(H)\subseteq\mathbb R^{E(H)}\), the row leverage of
\(e\in E(H)\) is
\[
 \ell_H(e)=(P_H)_{ee}.
\]
Throughout, \(\1\) denotes an all-ones vector of the dimension determined by context.

The line graph \(L(H)\) has vertex set \(E(H)\), with two vertices
adjacent when the corresponding hyperedges intersect.  A graph is a
\emph{cograph} if it has no induced path on four vertices.  Seinsche's
decomposition theorem states that every cograph with at least two
vertices is disconnected or has disconnected complement
\cite{Seinsche}.

\begin{lemma}[Cograph decomposition]\label{lem:cograph-decomposition}
If a graph \(G\) with at least two vertices has no induced \(P_4\),
then \(G\) is disconnected or \(\overline G\) is disconnected.
\end{lemma}

\begin{proof}
Assume that \(G\) is connected.  Fix \(x\in V(G)\), let
\(L_1=N_G(x)\), and let \(L_2=V(G)\setminus(L_1\cup\{x\})\).  A
shortest path in \(G\) is induced, so \(P_4\)-freeness implies that
every vertex is at distance at most two from \(x\).  If \(L_2\) is
empty, then \(x\) is isolated in \(\overline G\), and we are done.

Let \(C\) be a connected component of \(G[L_2]\).  It has a neighbor
in \(L_1\).  Every \(y\in L_1\) is either complete or anticomplete to
\(C\).  Indeed, if its adjacency changes along an edge \(uv\) of a
path in \(C\), with \(yu\in E(G)\) and \(yv\notin E(G)\), then
\(x,y,u,v\) induce a \(P_4\).  Let
\[
       A_C=\{y\in L_1:y\text{ is complete to }C\};
\]
this set is nonempty.

The sets \(A_C\), over the components of \(G[L_2]\), are linearly
ordered by inclusion.  Otherwise choose distinct components \(C,C'\),
vertices
\[
 y\in A_C\setminus A_{C'},\qquad
 y'\in A_{C'}\setminus A_C,
\]
and points \(z\in C\), \(z'\in C'\).  If \(yy'\in E(G)\), then
\(z,y,y',z'\) induce a \(P_4\); if \(yy'\notin E(G)\), then
\(z,y,x,y'\) do.

Choose a component \(C\) for which \(A=A_C\) is minimal.  Then \(A\)
is complete to every component of \(G[L_2]\).  It is also complete to
\(L_1\setminus A\): if \(a\in A\), \(y\in L_1\setminus A\), and
\(ay\notin E(G)\), choose \(z\in C\); the vertices \(z,a,x,y\) induce
a \(P_4\).  Finally, \(A\) is complete to \(x\).  Hence every edge
between \(A\) and \(V(G)\setminus A\) is present.  Both parts are
nonempty, so \(\overline G\) is disconnected.
\end{proof}

\begin{lemma}\label{lem:path-cograph}
A linear \(s\)-uniform hypergraph is \(P_4^s\)-free if and only if its
line graph is a cograph.
\end{lemma}

\begin{proof}
A copy of \(P_4^s\) gives an induced \(P_4\) in the line graph.
Conversely, suppose that four hyperedges induce a graph path.  Consecutive
edges intersect and nonconsecutive edges are disjoint.  By linearity the
three consecutive intersections are single vertices.  No two of these
vertices coincide, since that would make a nonconsecutive pair intersect.
The four edges therefore form \(P_4^s\).
\end{proof}

Identify an edge \(e\) with its incidence vector
\(x_e\in\{0,1\}^{V(H)}\).  Uniformity and linearity give
\[
x_e^{\mathsf T}x_f=
\begin{cases}
s,&e=f,\\
1,&ef\in E(L(H)),\\
0,&ef\notin E(L(H)).
\end{cases}
\]
Equivalently,
\[
             N(H)N(H)^{\mathsf T}=A(L(H))+sI.
\]

\begin{lemma}[No universal transversal]\label{lem:no-transversal}
Let \(\mathcal S\) be an \(S(2,s,s^2)\) contained in a linear
hypergraph \(H\).  If \(f\in E(H)\setminus E(\mathcal S)\), then \(f\)
cannot meet every block of \(\mathcal S\).
\end{lemma}

\begin{proof}
The design has \(s(s+1)\) blocks, and each point lies in \(s+1\)
blocks.  If \(f\) met every block, linearity would force each block to
meet \(f\) once.  Counting these incidences gives
\[
s(s+1)=\sum_{v\in f\cap V(\mathcal S)}(s+1)\le s(s+1).
\]
Thus every point of \(f\) belongs to \(V(\mathcal S)\).  Two points of
\(f\) lie in a unique block of the design, and that block would meet
\(f\) twice.  Linearity would force the block to equal \(f\), contrary
to the choice of \(f\).
\end{proof}

\subsection{The join-node rank formula}

Let \(F\) be a nonempty edge family whose line graph \(G\) is a
connected nontrivial cograph.  Let
\[
             F=F_1\dotcupunion\cdots\dotcupunion F_p,\qquad p\ge2,
\]
where the \(F_i\) are the connected components of \(\overline G\);
Lemma \ref{lem:cograph-decomposition} gives \(p\ge2\).
Thus every edge in \(F_i\) meets every edge in \(F_j\) when \(i\ne j\).
Put
\[
\begin{split}
m_i&=|F_i|,\\
\mathcal R_i&=\operatorname{span}\{x_e:e\in F_i\},\\
\rho_i&=\dim\mathcal R_i,\\
\mathcal D_i&=\operatorname{span}\{x_e-x_f:e,f\in F_i\}.
\end{split}
\]

\begin{lemma}[Orthogonal join decomposition]\label{lem:join-rank}
The spaces \(\mathcal D_1,\ldots,\mathcal D_p\) are mutually
orthogonal, and
\[
                         \dim\mathcal D_i=\rho_i-1.
\]
Let \(\mathcal D=\mathcal D_1\oplus\cdots\oplus\mathcal D_p\).
For \(e\in F_i\), let \(c_i\) be the orthogonal projection of \(x_e\)
onto \(\mathcal D^\perp\).  This is independent of \(e\), and
\[
        \1^{\mathsf T}c_i=s,\qquad
        c_i^{\mathsf T}c_j=1\quad(i\ne j).
\]
If
\[
             z=\bigl|\{i:\|c_i\|^2=1\}\bigr|,
\]
then
\[
\boxed{\rank N(F)=\sum_{i=1}^p\rho_i-\max\{z-1,0\}.}
\]
Moreover, all \(c_i\) of norm one are equal.
\end{lemma}

\begin{proof}
For \(i\ne j\), every row from \(F_i\) has inner product one with every
row from \(F_j\).  Hence
\[
             (x_e-x_f)^{\mathsf T}x_g=0
             \qquad(e,f\in F_i,\ g\in F_j).
\]
Thus \(\mathcal D_i\perp\mathcal R_j\), and the \(\mathcal D_i\) are
mutually orthogonal.

Every vector in \(\mathcal D_i\) has coordinate sum zero.  Conversely,
if
\[
             y=\sum_{e\in F_i}\alpha_e x_e\in\mathcal R_i,
             \qquad \1^{\mathsf T}y=0,
\]
then \(\sum_e\alpha_e=0\).  Fixing \(f\in F_i\), we obtain
\[
             y=\sum_{e\ne f}\alpha_e(x_e-x_f)\in\mathcal D_i.
\]
Thus \(\mathcal D_i\) is the kernel in \(\mathcal R_i\) of the
nonzero coordinate-sum functional, proving \(\dim\mathcal D_i=\rho_i-1\).

Rows in \(F_i\) differ by elements of \(\mathcal D_i\), so they have a
common projection \(c_i\).  Since every row in \(\mathcal R_i\) is
orthogonal to \(\mathcal D_j\) for \(j\ne i\),
\[
                         x_e-c_i\in\mathcal D_i.
\]
The coordinate-sum assertion and the stated inner products now follow from the incidence inner-product formula.

The total row space is the orthogonal direct sum
\[
\mathcal D\oplus\operatorname{span}\{c_1,\ldots,c_p\}.
\]
Suppose
\[
               \sum_{i=1}^p a_i c_i=0,\qquad S=\sum_i a_i.
\]
Taking coordinate sums gives \(S=0\).  Taking inner product with \(c_j\)
and using \(c_i^{\mathsf T}c_j=1\) for \(i\ne j\) gives
\[
              0=a_j(\|c_j\|^2-1)+S.
\]
Hence \(a_j=0\) unless \(\|c_j\|^2=1\).  If two \(c_i,c_j\) have norm
one, then \(c_i^{\mathsf T}c_j=1\), and therefore \(c_i=c_j\).
The relation space among the \(c_i\) has dimension \(z-1\) when
\(z\ge1\), and dimension zero when \(z=0\).  Combining this with
the dimension formula for the \(\mathcal D_i\) and this orthogonal decomposition proves the stated rank formula.
\end{proof}

A co-component \(F_i\) with \(\|c_i\|^2=1\) will be called
\emph{balanced}.

\begin{lemma}[Balanced co-components]\label{lem:balanced}
Every balanced co-component \(F_i\) satisfies
\[
                         \rho_i=m_i,\qquad m_i\ge s.
\]
If \(m_i=s\), then its \(s\) edges are pairwise disjoint and
\[
                         c_i=\frac1s\sum_{e\in F_i}x_e.
\]
\end{lemma}

\begin{proof}
Fix a balanced component and abbreviate \(c=c_i\), \(m=m_i\).  Since
\(x_e-c\in\mathcal D_i\) and \(c\perp\mathcal D_i\),
\[
                         c^{\mathsf T}x_e=\|c\|^2=1.
\]
Let \(Q\) be the disjointness graph on \(F_i\).  It is connected because
\(F_i\) is a complement component.  Put \(y_e=x_e-c\).  By the incidence inner-product formula and \(c^{\mathsf T}x_e=1\),
\[
y_e^{\mathsf T}y_f=
\begin{cases}
s-1,&e=f,\\
0,&e\ne f\text{ and }e\cap f\ne\varnothing,\\
-1,&e\cap f=\varnothing.
\end{cases}
\]
Thus the Gram matrix of the \(y_e\) is
\[
                         (s-1)I-A(Q).
\]

The vector \(c\) lies in the affine hull of the rows in \(F_i\):
\(x_e-c\) is a linear combination of row differences.  Thus
\[
             c=\sum_{e\in F_i}\lambda_e x_e,
             \qquad \sum_e\lambda_e=1.
\]
Equivalently, \(\sum_e\lambda_e y_e=0\), so the displayed Gram matrix is singular.  It is
positive semidefinite, hence every eigenvalue of \(A(Q)\) is at most
\(s-1\), and \(s-1\) is an eigenvalue.  Since \(Q\) is connected,
Perron--Frobenius theory gives a one-dimensional kernel for this Gram matrix,
spanned by a vector with all entries of one sign.  In particular
\[
                         s-1\le m-1,
\]
so \(m\ge s\).

If \(\sum_e a_e x_e=0\), then coordinate sums give \(\sum_e a_e=0\), and
therefore \(\sum_e a_e y_e=0\).  A nonzero kernel vector of this Gram matrix has all
coordinates of one sign, whereas \(\sum_e a_e=0\).  Hence every \(a_e\)
vanishes and the rows are independent.

If \(m=s\), the connected graph \(Q\) has order \(s\) and spectral
radius \(s-1\), so \(Q=K_s\).  The edges are pairwise disjoint.  The
kernel of \(sI-J\) is spanned by the all-ones vector, which gives the claimed formula for \(c\).
\end{proof}

\begin{lemma}[Parallel-class bound]\label{lem:class-bound}
At a fixed join node, at most \(s+1\) balanced co-components have size
exactly \(s\).
\end{lemma}

\begin{proof}
Let \(F_1,\ldots,F_a\) be the balanced co-components of size \(s\).
By Lemma \ref{lem:join-rank}, their vectors \(c_i\) are equal to a
common vector \(c\).  By Lemma \ref{lem:balanced}, every \(F_i\)
consists of \(s\) pairwise disjoint \(s\)-sets and
\[
                         c=\frac1s\sum_{e\in F_i}x_e.
\]
Consequently every \(F_i\) partitions the same set \(U\) of \(s^2\)
points, and \(c=s^{-1}\1_U\).  Edges in distinct classes meet once.

For each \(i\), let
\[
 \mathcal W_i=
 \operatorname{span}\left\{x_e-\frac1s\1_U:e\in F_i\right\}.
\]
This space has dimension \(s-1\) and lies in \(\1_U^\perp\).  If
\(e\in F_i\), \(f\in F_j\), and \(i\ne j\), then
\[
\left(x_e-\frac1s\1_U\right)^{\mathsf T}
\left(x_f-\frac1s\1_U\right)=1-1-1+1=0.
\]
The \(\mathcal W_i\) are mutually orthogonal subspaces of the
\((s^2-1)\)-dimensional space \(\1_U^\perp\).  Hence
\[
                    a(s-1)\le s^2-1=(s-1)(s+1),
\]
and \(a\le s+1\).
\end{proof}

\begin{theorem}[Cograph incidence rank]\label{thm:rank}
Let \(H\) be a finite linear \(s\)-uniform hypergraph, \(s\ge2\), whose
line graph is a cograph.  Then
\[
\boxed{(s+1)\rank_{\R}N(H)\ge s|E(H)|.}
\]
Equality holds if and only if every edge-containing component of \(H\)
is a Steiner system \(S(2,s,s^2)\).
\end{theorem}

\begin{proof}
Write
\[
       \rho(H)=\rank N(H),\qquad
       \sigma(H)=(s+1)\rho(H)-s|E(H)|.
\]
We induct on \(m=|E(H)|\).  The cases \(m=0,1\) are immediate, with
\(\sigma(H)=1\) when \(m=1\).

If \(L(H)\) is disconnected, distinct line-graph components have
disjoint point supports.  Rank and \(\sigma\) are additive, and the
induction hypothesis applies.

Now suppose that \(L(H)\) is connected and \(m\ge2\).  Use the root join decomposition and notation above,
and let \(\mathcal B\) be the set of balanced indices, with
\(z=|\mathcal B|\).  If \(z\le1\), Lemma \ref{lem:join-rank} and
induction give
\[
           \rho(H)=\sum_i\rho_i
           \ge\frac{s}{s+1}\sum_i m_i
           =\frac{sm}{s+1}.
\]

Suppose \(z\ge2\).  Balanced components have full row rank by
Lemma \ref{lem:balanced}.  Therefore
\[
\rho(H)
=\sum_{i\notin\mathcal B}\rho_i
+\sum_{i\in\mathcal B}m_i-z+1
\ge
\frac{s}{s+1}\sum_{i\notin\mathcal B}m_i
+\sum_{i\in\mathcal B}m_i-z+1.
\]
It remains to prove
\[
                 \sum_{i\in\mathcal B}m_i
                 \ge(s+1)(z-1).
\]
Let \(a\) be the number of balanced components of size \(s\).
Every other balanced component has size at least \(s+1\), and
Lemma \ref{lem:class-bound} gives \(a\le s+1\).  Thus
\[
\sum_{i\in\mathcal B}m_i
\ge as+(z-a)(s+1)
=(s+1)z-a
\ge(s+1)(z-1).
\]
This proves the incidence-rank inequality.

We classify equality by the same induction.  At a disconnected line
graph, equality holds precisely when it holds in every component.
Suppose the line graph is connected.  If \(z\le1\), the exact formula
gives
\[
                         \sigma(H)=\sum_i\sigma(F_i).
\]
Equality would force equality in every \(F_i\).  By induction, one
\(F_i\) contains an \(S(2,s,s^2)\).  An edge in another joined
co-component would meet every block of that design, contrary to
Lemma \ref{lem:no-transversal}.  Hence equality is impossible.

If \(z\ge2\), the exact formula is
\[
\sigma(H)=
\sum_{i\notin\mathcal B}\sigma(F_i)
+\sum_{i\in\mathcal B}m_i-(s+1)(z-1).
\]
The estimate above gives
\[
\sum_{i\in\mathcal B}m_i-(s+1)(z-1)\ge s+1-a\ge0.
\]
If equality holds, then \(a=s+1\).  These \(s+1\) minimal balanced
components are parallel classes on a common set \(U\) of \(s^2\)
points.  The number of unordered point-pairs covered by their blocks is
\[
                         s(s+1)\binom{s}{2}
                         =\binom{s^2}{2}.
\]
Linearity prevents double counting, so every pair of points of \(U\)
lies in exactly one block.  The blocks form an \(S(2,s,s^2)\).
Lemma \ref{lem:no-transversal} excludes any additional joined edge.

Conversely, if \(N\) is the incidence matrix of \(S(2,s,s^2)\), then
\[
                         N^{\mathsf T}N=sI+J,
\]
so \(N\) has full column rank \(s^2\).  Since the design has
\(s(s+1)\) blocks, equality holds in the incidence-rank inequality.
\end{proof}

\begin{corollary}[Linear four-edge paths]\label{cor:path}
For every \(s\ge2\) and every finite linear \(P_4^s\)-free
\(s\)-uniform hypergraph \(H\),
\[
                         |E(H)|\le\frac{s+1}{s}|V(H)|.
\]
Equality holds if and only if \(H\) has no isolated vertices and its
edge-containing components are vertex-disjoint copies of
\(S(2,s,s^2)\).
\end{corollary}

\begin{proof}
By Lemma \ref{lem:path-cograph}, the line graph is a cograph.  Theorem
\ref{thm:rank} gives
\[
             \frac{s|E(H)|}{s+1}\le\rank N(H)\le |V(H)|.
\]
Equality forces equality in both inequalities, and the statement
follows from Theorem \ref{thm:rank}.  The converse is immediate.
\end{proof}

\section{Finite nets and defect one}\label{sec:nets}

Standard finite-geometric terminology is used \cite{Dembowski1968}, with
all definitions and arguments needed here stated explicitly.

An \(a\)-\emph{net of order \(s\)} is a linear \(s\)-uniform incidence
structure on \(s^2\) points with \(a\) parallel classes.  Each class
consists of \(s\) disjoint lines partitioning the point set, and lines
from different classes meet exactly once.  The case \(a=s+1\) is an
affine plane.  Put
\[
                         d=s+1-a,
\]
the deficiency of the net.

\begin{theorem}[Finite-net spectrum and defect]\label{thm:net-spectrum}
Let \(\mathcal N\) be an \(a\)-net of order \(s\), with incidence matrix
\(N\).  Then:
\begin{enumerate}
\item \(L(\mathcal N)=K_{s,\ldots,s}\), with \(a\) parts;
\item the spectrum of \(NN^{\mathsf T}\) is
      \[
      sa\ (1),\qquad s\ \bigl(a(s-1)\bigr),\qquad
      0\ (a-1);
      \]
\item every point has degree \(a\), and the eigenvalues of
      \(a^{-1}NN^{\mathsf T}\) are
      \[
      s\ (1),\qquad \frac sa\ \bigl(a(s-1)\bigr),\qquad
      0\ (a-1);
      \]
\item
      \[
      \rank N=1+a(s-1);
      \]
\item every line \(B\) has row leverage
      \[
      \ell_{\mathcal N}(B)=\frac{1+a(s-1)}{as};
      \]
\item the defects are
      \[
      \boxed{R_s(\mathcal N)=d,}
      \qquad
      \boxed{C(\mathcal N)=(s-1)d,}
      \]
      and
      \[
      \boxed{(s+1)|V(\mathcal N)|-s|E(\mathcal N)|=s^2d.}
      \]
\end{enumerate}
\end{theorem}

\begin{proof}
The adjacency spectrum of the complete \(a\)-partite graph with \(s\)
vertices in every part is
\[
      s(a-1)\ (1),\qquad -s\ (a-1),\qquad
      0\ \bigl(a(s-1)\bigr).
\]
Since \(NN^{\mathsf T}=sI+A(L(\mathcal N))\), the first spectrum and
the rank formula follow.  Every point lies on one line from every
parallel class, so the normalized Gram matrix is
\(a^{-1}NN^{\mathsf T}\).  The Gram pattern is invariant under every
permutation within and among the equal parts, so the leverage diagonal
is constant; its value is rank divided by the number \(as\) of rows.
Finally,
\[
R_s(\mathcal N)
=(s+1)(1+a(s-1))-s(as)
=s+1-a=d.
\]
The remaining identities follow from the definitions and \eqref{eq:rank-column-split}.
\end{proof}

Deleting one parallel class from an affine plane therefore costs one
incidence-rank-defect unit, \(s-1\) column-rank-defect units, and \(s^2\)
units of path Tur\'an slack.  This is the discrete defect scale behind
the equality theorem.

A single edge also has incidence-rank defect one, so the hypothesis \(|E(F)|\ge2\) is necessary in the characterization below.

\begin{theorem}[First positive cograph defect]\label{thm:first-defect}
Let \(F\) be a connected linear \(s\)-uniform hypergraph whose line
graph is a cograph, and suppose \(|E(F)|\ge2\).  Then
\[
        R_s(F)=1
        \quad\Longleftrightarrow\quad
        F\text{ is an \(s\)-net of order \(s\)}.
\]
\end{theorem}

\begin{proof}
Use the root join decomposition above.  Let \(\mathcal B\) be the
balanced indices, \(z=|\mathcal B|\), and let \(a\) count the balanced
co-components of size \(s\).

If \(z\le1\), Lemma \ref{lem:join-rank} gives
\[
                         R_s(F)=\sum_i R_s(F_i).
\]
If this sum were one, then, since the join has at least two nonempty
co-components, some other co-component would have defect zero.
By Theorem \ref{thm:rank}, it contains an affine plane.  Any edge in
the positive-defect co-component would meet every line of that plane,
contrary to Lemma \ref{lem:no-transversal}.  Hence \(z\ge2\).

The exact rank formula gives
\[
R_s(F)=
\sum_{i\notin\mathcal B}R_s(F_i)
+\sum_{i\in\mathcal B}m_i-(s+1)(z-1).
\]
Both terms are nonnegative integers, and
\[
\sum_{i\in\mathcal B}m_i-(s+1)(z-1)
\ge s+1-a.
\]
A zero value in the preceding lower bound would force \(a=s+1\).  Those minimal components
already form an affine plane, and Lemma \ref{lem:no-transversal}
excludes every further joined edge.  The resulting defect would be
zero, not one.  Therefore the join contribution is one, and every
nonbalanced term has defect zero.  A nonbalanced component cannot
actually occur, since it would contain an affine plane joined to an
edge in a balanced component.  Thus all root co-components are
balanced.

The same lower bound gives \(a\ge s\).  The case \(a=s+1\) again gives an
affine plane and defect zero, so \(a=s\).  These \(s\) minimal
components are \(s\) parallel classes on a common \(s^2\)-point set
\(U\); they form an \(s\)-net.

It remains to exclude another balanced co-component.  Every edge \(B\)
in such a component meets all \(s\) lines of each existing class.
A single class partitions \(U\), so the \(s\)-set \(B\) is contained in
\(U\) and is a transversal of every class.  Define a graph \(Q\) on
\(U\) by joining two points not lying together on a net line.  Each
point has
\[
                    s^2-1-s(s-1)=s-1
\]
neighbors in \(Q\).  A transversal containing a point \(x\) consists
of \(x\) and all \(s-1\) of its neighbors in \(Q\).  Hence two distinct
transversals cannot intersect.  Every extra co-component is therefore a
family of pairwise disjoint \(s\)-sets contained in the \(s^2\)-point
set \(U\).  It has at most \(s\) members, while
Lemma \ref{lem:balanced} gives at least \(s\).  It therefore has size
\(s\), making it another minimal
component and contradicting \(a=s\).  Thus no extra component exists.

Conversely, Theorem \ref{thm:net-spectrum} with \(a=s\) gives
\(R_s=1\).
\end{proof}

\begin{remark}
Theorem \ref{thm:first-defect} is an integral stability statement, not
an edit-distance theorem.  A global stability result would also have
to control how these defect units distribute through the cograph
cotree.
\end{remark}

\section{Join-node defect formula}\label{sec:toll}

The join formula used in Theorem~\ref{thm:rank} gives an exact recursion
for \(R_s\).  At a join node, retain the notation
\[
 E(F)=F_1\dotcupunion\cdots\dotcupunion F_p,\qquad
 m_i=|F_i|,\qquad \rho_i=\operatorname{rank}N(F_i),
\]
and let \(\mathcal B\) be the set of balanced indices from
Lemma~\ref{lem:join-rank}.  Put
\[
 z=|\mathcal B|,
 \qquad
 a=\bigl|\{i\in\mathcal B:m_i=s\}\bigr|.
\]

\begin{lemma}[Join-node defect formula]\label{lem:cotree-toll}
At a join node,
\[
 R_s(F)=
 \begin{cases}
 \displaystyle\sum_{i=1}^p R_s(F_i),&z\le1,\\[2mm]
 \displaystyle\sum_{i\notin\mathcal B}R_s(F_i)+J_s(F),&z\ge2,
 \end{cases}
\]
where
\[
 J_s(F)=\sum_{i\in\mathcal B}m_i-(s+1)(z-1)
        =s+1-z+\sum_{i\in\mathcal B}(m_i-s).
\]
Moreover
\[
 J_s(F)\ge s+1-a\ge0.
\]
At a disjoint-union node, \(R_s\) is additive.
\end{lemma}

\begin{proof}
Lemma~\ref{lem:join-rank} gives
\[
 \operatorname{rank}N(F)=
 \sum_i\rho_i-\max\{z-1,0\}.
\]
Substitution in the definition of \(R_s\) gives the two displayed
formulas.  If \(z\ge2\), every balanced child has at least \(s\) edges
by Lemma~\ref{lem:balanced}, and every balanced child not counted by
\(a\) has at least \(s+1\) edges.  Hence
\[
 \sum_{i\in\mathcal B}(m_i-s)\ge z-a,
\]
which gives the lower bound for \(J_s(F)\).  Additivity at a union
node follows because distinct line-graph components have disjoint point
supports.
\end{proof}

In particular, if \(R_s(F)\le d\) and a join node has \(z\ge2\), then
\begin{equation}\label{eq:minimal-balanced}
 a\ge s+1-d.
\end{equation}
The \(a\) minimal balanced children are parallel classes on one
\(s^2\)-point set by Lemma~\ref{lem:class-bound}.  They therefore form
an \(a\)-net of deficiency at most \(d\).

The same calculation gives the rank of an incomplete collection of parallel
classes.  This will be used after a net has been completed.

\begin{lemma}[Rank and leverage in affine-plane subsystems]\label{lem:affine-packet}
Let \(\Pi\) be an affine plane of order \(s\).  Let \(F\) consist of
nonempty subfamilies \(\mathcal A_1,\ldots,\mathcal A_p\) of distinct
parallel classes of \(\Pi\), with \(m_i=|\mathcal A_i|\le s\).  Put
\[
 M=\sum_{i=1}^p m_i,\qquad
 z=\bigl|\{i:m_i=s\}\bigr|.
\]
Then
\[
 \operatorname{null}N(F)=\max\{z-1,0\},
 \qquad
 \operatorname{rank}N(F)=M-\max\{z-1,0\}.
\]
If \(z\ge1\), every row in a full class has leverage
\[
 1-\frac{z-1}{zs},
\]
and every row in an incomplete class has leverage \(1\).  If \(z=0\),
all rows have leverage \(1\).
\end{lemma}

\begin{proof}
The line graph is \(K_{m_1,\ldots,m_p}\), so
\[
 N(F)N(F)^{\mathsf T}=sI+A(K_{m_1,\ldots,m_p}).
\]
Let \(x=(x_B)\) lie in the kernel of this Gram matrix and put
\[
 S=\sum_{i=1}^p\sum_{B\in\mathcal A_i}x_B.
\]
Subtracting two kernel equations in the same class shows that
\(x_B=c_i\) on \(\mathcal A_i\).  The remaining equations are
\[
 (s-m_i)c_i+S=0\qquad(1\le i\le p).
\]
If one class is full, then \(S=0\).  The constants on incomplete
classes vanish, while those on the \(z\) full classes satisfy one
linear relation, so the nullity is \(z-1\).  If no class is full,
substitution into the definition of \(S\) gives
\[
 S=-S\sum_i\frac{m_i}{s-m_i},
\]
hence \(S=0\) and all \(c_i=0\).

For \(z\ge1\), the kernel consists of class-constant vectors supported
on the full classes whose class constants sum to zero.  Its orthogonal
projector has diagonal \((z-1)/(zs)\) on each row of a full class and
zero on every incomplete-class row.  Subtracting from the identity
gives the leverage formula.
\end{proof}

\section{Completion and the bounded-defect classification}\label{sec:classification}

A net of order \(s\) and deficiency \(h\) has \(s+1-h\) parallel
classes.  We use two classical completion results.  Bruck proved that
if \(s>(h-1)^2\), then an embedding of such a net in an affine plane,
when it exists, is unique and every transversal of the net is a line
of that completion \cite{Bruck1963}.  His completion theorem also gives
existence under
\[
 s>P(h-1),\qquad
 P(x)=\frac12x^4+x^3+x^2+\frac34x.
\]
Metsch proved the sharper sufficient condition
\[
 3s>8h^3-18h^2+8h+4
\]
for completion \cite{Metsch1991}.

\begin{definition}[Completion hypothesis]\label{def:comp}
For \(0\le d\le s-1\), write \(\operatorname{Comp}(s,d)\) if
\begin{enumerate}
\item every net of order \(s\) and deficiency at most \(d\) embeds in
      an affine plane of order \(s\);
\item \(s>(d-1)^2\).
\end{enumerate}
\end{definition}

Thus \(s>P(d-1)\) implies \(\operatorname{Comp}(s,d)\).  The result of
Metsch implies the same conclusion whenever
\[
 3s>8d^3-18d^2+8d+4.
\]

\begin{theorem}[Completion-range classification]\label{thm:classification}
Let \(F\) be a nonempty connected linear \(s\)-uniform hypergraph whose
line graph is a cograph.  Suppose
\[
 R_s(F)=d\le s-1
\]
and \(\operatorname{Comp}(s,d)\) holds.  Exactly one of the following
occurs.
\begin{enumerate}
\item \emph{Full row rank.}
      \[
      |E(F)|=d,\qquad \operatorname{rank}N(F)=d.
      \]
\item \emph{Affine-plane subsystem.}  There are a unique affine plane
      \(\Pi\) of order \(s\) and an integer \(h\), \(0\le h\le d\),
      such that \(F\subseteq\Pi\), \(F\) contains every line in
      \(s+1-h\) parallel classes of \(\Pi\), and \(F\) contains exactly
      \[
      M=d-h
      \]
      additional lines from the remaining \(h\) classes.  If \(d>0\),
      then \(1\le h\le d\); if \(d=0\), then \(h=0\) and \(F=\Pi\).
\end{enumerate}
Conversely, every system in either class has incidence-rank defect \(d\).
\end{theorem}

\begin{proof}
The identity
\[
 R_s(F)=|E(F)|-(s+1)\operatorname{null}N(F)
\]
shows that a full-row-rank system has \(|E(F)|=d\).  Conversely, if
\(|E(F)|=d\le s-1\), its incidence rows are independent.  Indeed, the
Gram matrix of any \(d\) rows is \(sI+A\), where \(A\) is a
\(0\)-\(1\) matrix with zero diagonal and every row sum at most
\(d-1<s\).  Strict diagonal dominance makes this Gram matrix positive
definite.

Assume now that the rows are dependent.  We argue by induction on
\(d\).  The case \(d=0\) is Theorem~\ref{thm:rank}.  At the root join, use the notation of Lemma~\ref{lem:cotree-toll}.  The
first claim is \(z\ge2\).

Suppose \(z\le1\).  Then
\[
 d=\sum_i R_s(F_i).
\]
No line-graph component of any \(F_i\) can have defect zero: by
Theorem~\ref{thm:rank} it would contain an affine plane, while any edge
in another root co-component would meet every line of that plane,
contrary to Lemma~\ref{lem:no-transversal}.  Hence every nonempty
line-graph component occurring among the \(F_i\) has positive defect
strictly smaller than \(d\).

If some \(F_i\) had a row-dependent line-graph component \(C\), the
induction hypothesis would place \(C\) in a unique affine plane and
would give it at least
\[
 s+1-R_s(C)\ge s+1-d\ge2
\]
full directions on a point set \(U\).  Choose an edge \(g\) in another
root co-component.  Since \(g\) meets all \(s\) lines of a full class
of \(C\), its \(s\) distinct intersection points exhaust \(g\), so
\(g\subseteq U\).  The co-component \(F_i\) cannot have a second
line-graph component, since that component is point-disjoint from \(U\)
while \(g\) must meet all of its edges.  Thus \(F_i=C\).  The complement
of the line graph of a subset of an affine plane is the disjoint union
of its nonempty directions; it is disconnected because \(C\) has at
least two full directions.  This contradicts the definition of \(F_i\)
as a root co-component.  Therefore every \(F_i\) has full row rank.
Lemma~\ref{lem:join-rank} then makes \(F\) full row rank, a
contradiction.  Hence \(z\ge2\).

By \eqref{eq:minimal-balanced},
\[
 a\ge s+1-d\ge2.
\]
The \(a\) minimal balanced co-components form an \(a\)-net
\(\mathcal N\) on a common point set \(U\).  Put
\[
 h=s+1-a\le d.
\]
Every edge outside these classes meets all \(s\) lines of one retained
class.  Its \(s\) distinct intersections exhaust the edge, so the edge
lies in \(U\) and is a transversal of \(\mathcal N\).

The completion hypothesis embeds \(\mathcal N\) in an affine plane
\(\Pi\).  Since \(s>(h-1)^2\), Bruck's uniqueness and transversal
theorem makes \(\Pi\) unique and places every transversal of
\(\mathcal N\) on a line of \(\Pi\).  Thus every edge of \(F\) is a
line of this same affine plane.

For a subset of an affine plane, the connected components of the
complement of the line graph are precisely its nonempty directions.
A balanced direction has at least \(s\) lines by
Lemma~\ref{lem:balanced}; it therefore consists of all \(s\) lines of
that direction.  Hence the \(a=s+1-h\) minimal balanced components are
exactly the full directions of \(F\).  Let \(M\) be the total number
of retained lines in the remaining \(h\) directions.  By
Lemma~\ref{lem:affine-packet},
\[
 R_s(F)=s+1-a+M=h+M.
\]
Since the left side is \(d\), we have \(M=d-h\).

Conversely, an affine-plane subsystem in the statement has
\(a=s+1-h\) full directions and \(M=d-h\) additional rows.
Lemma~\ref{lem:affine-packet} gives \(R_s=h+M=d\).  The independent
case was handled at the beginning.
\end{proof}

\begin{corollary}[Fixed and growing defect]\label{cor:bounded}
Fix \(D\).  If \(s>P(D-1)\), Theorem~\ref{thm:classification}
classifies every connected cograph line system with \(R_s\le D\).
More generally, the classification holds for a defect \(d=d(s)\)
whenever
\[
 3s>8d^3-18d^2+8d+4;
\]
in particular it holds for \(d=o(s^{1/3})\).
\end{corollary}

For \(d=1\), the geometric alternative in
Theorem~\ref{thm:classification} is an affine plane with one parallel
class deleted.  This sharpens Theorem~\ref{thm:first-defect} in the
completion range.

\begin{corollary}[Global bounded-defect decomposition]\label{cor:global}
Let \(H\) be a finite linear \(s\)-uniform hypergraph with cograph line
graph.  Suppose
\[
 R_s(H)\le D\le s-1
\]
and \(\operatorname{Comp}(s,D)\) holds.  After isolated vertices are
removed, \(H\) is the vertex-disjoint union of affine planes of order
\(s\), full-row-rank components, and affine-plane subsystems as in
Theorem~\ref{thm:classification}.  The positive component defects sum
to at most \(D\).

Deleting at most \(D\) edges turns \(H\) into a vertex-disjoint union
of finite nets and affine planes.  Alternatively, deleting the
full-row-rank components and completing every geometric component requires
at most \(sD\) added lines.  For one geometric component of defect
\(d\) and parameter \(h\), the exact number of added lines is
\[
 hs-(d-h)=h(s+1)-d.
\]
\end{corollary}

\begin{proof}
Distinct line-graph components have disjoint point supports, and
\(R_s\) is additive.  Apply Theorem~\ref{thm:classification}
componentwise.  In a geometric component, delete its \(d-h\) partial
direction lines to leave a net; in a full-row-rank component, delete all
\(d\) edges.  The total number deleted is at most the total positive
defect.

For the completion statement, add every missing line from the \(h\)
directions of a geometric component.  This requires
\(hs-(d-h)=h(s+1)-d\le sd\) lines.  Summing over components gives the
claim.
\end{proof}

\section{Consequences of the classification}\label{sec:consequences}

\begin{corollary}[Column defect and leverage]\label{cor:quantization}
In the geometric alternative of Theorem~\ref{thm:classification}, with
\(d>0\),
\[
 |V(F)|=s^2,
\]
\[
 \operatorname{rank}N(F)
 =1+(s+1-h)(s-1)+(d-h),
\]
and
\[
 C(F)=hs-d.
\]
The extremal slack is
\[
 (s+1)|V(F)|-s|E(F)|
 =s\bigl((s+1)h-d\bigr).
\]
Every partial-direction row has leverage \(1\), while every row in a
full parallel class has leverage
\[
 1-\frac{s-h}{s(s+1-h)}.
\]
\end{corollary}

\begin{proof}
There are \(a=s+1-h\) full directions and \(M=d-h\) partial lines.
Lemma~\ref{lem:affine-packet} gives
\[
 \operatorname{rank}N(F)=as+M-(a-1)
 =1+a(s-1)+M.
\]
The formulas follow by substitution.  The leverage values are the two
cases of Lemma~\ref{lem:affine-packet}.
\end{proof}

For fixed \(d\), the possible column defects of a dependent connected
component are
\[
 s-d,\ 2s-d,\ \ldots,\ ds-d,
\]
subject to realization of the corresponding partial directions in the
unique affine completion.  Rows of leverage one are exactly the
partial-direction rows and therefore identify the edges that must be
deleted to recover the net core.

\begin{corollary}[Affine-plane rank certificate]\label{cor:plane-certificate}
Assume the hypotheses of Theorem~\ref{thm:classification}.  If a
connected component with incidence-rank defect \(d\) is row dependent, then
an affine plane of order \(s\) exists.  Conversely, deleting \(d\)
parallel classes from an affine plane of order \(s\) gives a connected
row-dependent cograph line system of defect \(d\), for
\(0\le d\le s-1\).
\end{corollary}

Consequently, at an order \(s\) for which no affine plane exists,
a connected row-dependent cograph incidence system cannot lie in the
Metsch completion range.  If \(d=R_s(F)\), then necessarily
\[
 3s\le 8d^3-18d^2+8d+4,
\]
so along non-plane orders
\[
 R_s(F)=\Omega(s^{1/3})
\]
for connected row-dependent systems.

\section{Concluding remarks}\label{sec:conclusion}

The incidence-rank defect is integer valued and additive across line-graph
components.  Theorem~\ref{thm:classification} shows that, throughout
the finite-net completion range, every dependent connected component
is obtained from a single affine plane by retaining all lines in most
directions and a bounded number of lines in the remaining directions.
At each join node, the cotree decomposition identifies the complete
directions, while rows of leverage one identify the additional lines.

Beyond the completion range, the affine-plane description requires
additional hypotheses.  A natural problem is to determine whether a connected
cograph line system with small \(R_s\) must contain large net substructures
and whether deleting \(O(R_s)\) rows leaves a disjoint union of nets that are
inclusion-maximal within the original system.  The real-rank argument used
here relies on positive semidefinite Gram matrices and Perron--Frobenius
theory; finite-characteristic incidence rank is a separate problem.


\begin{thebibliography}{99}

\bibitem{AdakVerma2026}
R. Adak and P. Verma,
``Linear Tur\'an numbers of uniform hypertrees,''
arXiv:2607.16854 (2026).

\bibitem{Bruck1963}
R. H. Bruck,
``Finite nets. II. Uniqueness and imbedding,''
\emph{Pacific Journal of Mathematics} 13 (1963), 421--457.

\bibitem{Dembowski1968}
P. Dembowski,
\emph{Finite Geometries}.
Springer-Verlag, Berlin, 1968.

\bibitem{Metsch1991}
K. Metsch,
``Improvement of Bruck's completion theorem,''
\emph{Designs, Codes and Cryptography} 1 (1991), 99--116.

\bibitem{Seinsche}
D. Seinsche,
``On a property of the class of $n$-colorable graphs,''
\emph{Journal of Combinatorial Theory, Series B} 16 (1974), 191--193.

\end{thebibliography}
\end{document}